\documentclass[12pt,psamsfonts]{amsart}
\usepackage{amsmath}
\usepackage{amsthm}
\usepackage{amssymb}
\usepackage{amscd}
\usepackage{amsfonts}
\usepackage{amsbsy}
\usepackage{graphicx}
\usepackage[dvips]{psfrag}
\usepackage{array}
\usepackage{setspace}

\newcolumntype{L}{>{\displaystyle}l}
\newcolumntype{C}{>{\displaystyle}c}
\newcolumntype{R}{>{\displaystyle}r}

\newcommand{\R}{\ensuremath{\mathbb{R}}}

\newcommand{\T}{\theta}

\newcommand{\x}{\mathbf{x}}
\newcommand{\y}{\mathbf{y}}
\newcommand{\z}{\mathbf{z}}
\newcommand{\sgn}{\mathrm{sign}}

\def\p{\partial}
\def\e{\varepsilon}

\newtheorem {theorem} {Theorem} 

\newtheorem {corollary} [theorem] {Corollary}

\newtheorem {remark} {Remark}

\begin{document}
\singlespacing 

\title[Perturbed damped pendulum: finding periodic solutions]
{Perturbed damped pendulum:\\finding periodic solutions}

\author[D.D. Novaes]
{Douglas D. Novaes$^1$}

\address{$^1$ Departamento de Matematica, Universidade
Estadual de Campinas, Caixa Postal 6065, 13083--859, Campinas, SP,
Brasil} \email{ddnovaes@gmail.com}

\subjclass[2010]{37G15, 37C80, 37C30}

\keywords{averaging theory, periodic solutions, damped pendulum}

\maketitle

\begin{abstract}
Using the damped pendulum system we introduce the averaging method to study the periodic solutions of a dynamical system with small perturbation. We provide sufficient conditions for the existence of periodic
solutions with small amplitude of the non--linear perturbed damped pendulum. 
The averaging theory provides a useful means to study dynamical systems, accessible to Master and PhD students.
\end{abstract}

\section{Introduction}\label{sec1}

Systems derived from the pendulum give to students important and practical examples of dynamical systems. For instance, we can see the \textit{weight-driven pendulum clocks} which had its historical and dynamical aspects studied by Denny in a recent paper \cite{D}. This system has been revisited by Llibre and Teixeira in \cite{LT}, and using some simple techniques, from averaging theory, they got the same results. These systems, some times, has also been used to introduce mathematical concepts of classical mechanics, as we can see in \cite{MCPAP}.

\smallskip

In this paper we attempt to use a simple physical system, as the damped pendulum, to introduce some concepts and techniques of the important and useful averaging theory, which can be used to study the periodic solutions of dynamical systems. For instance, in \cite{LNT1}, Llibre, Novaes and Teixeira have used the averaging theory to provide sufficient conditions for the existence of periodic solutions of the planar double pendulum with small oscillations perturbed by non--linear functions. 

\section{The Damped Pendulum}\label{secDP}

We consider a system composed of a point mass $m$ moving in the plane, under gravity force, such that the distance between the point mass $m$ and a given point $P$ is fixed and equal to $l$. We also consider that the motion of the particle suffers a resistance proportional to its velocity. This system is called \textit{Damped Pendulum}.

\smallskip

The position of the pendulum is determined by the angle
$\T$ shown in Figure \ref{fig1}. The equation of
motion of this system is given by
\begin{equation}\label{s1}
\ddot{\T}=-a\sin(\T)-b\dot{\T},
\end{equation}
where $a>0$ and $b>0$ are real parameters, with $a=g/l$, $g$ the acceleration of the gravity, $l$ the length of the rod and $b$ the damping coefficient.

\begin{figure}[h]
\psfrag{A}{$l$}
\psfrag{B}{$\theta$}
\psfrag{C}{$m$}
\psfrag{G}{$g$}
\psfrag{P}{$P$}
\includegraphics[width=4cm]{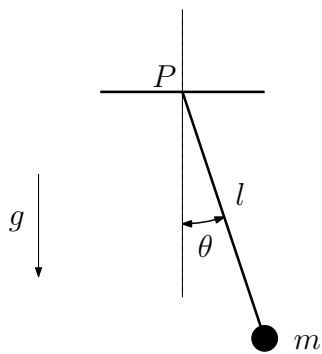}
\vskip 0cm \centerline{} \caption{\small \label{fig1} Pendulum.}
\end{figure}

\smallskip

There are many other kinds of resistance that the particle motion can suffer, providing many different dynamical behaviors. For instance, the {\it Coulomb Friction} introduces a discontinuous term in the equation of motion \eqref{s1}. For more details about this last issue, see the book of Andronov, Vitt and Khaikin \cite{AVK}.

\smallskip

The system \eqref{s1} has the following eigenvalues
\[
\begin{array}{ccc}
\lambda_1=\dfrac{-b+\sqrt{b^2-4 a}}{2}& \textrm{and} & \lambda_2=\dfrac{-b-\sqrt{b^2-4 a}}{2}.
\end{array}
\]

\smallskip

In the qualitative theory of dynamical systems, a singularity of a differential system is called hyperbolic if their eigenvalues has non--zero real components. In this case, applying the classical {\it Hartman-Grobman Theorem} we can study the local behavior of the system looking to the linearized system. Here, the behavior of a dynamical system can be informally understood as how the phase portrait looks like. For a general introduction to qualitative theory of dynamical systems see for instance the book of Arrowsmith and Place \cite{AP}.

\smallskip

The linearized equation of motion of the damped pendulum is given by
\begin{equation}\label{sl1}
\ddot{\T}=-a\T-b\dot{\T}.
\end{equation}
Note that for $b^2\geq 4a$ the eigenvalues $\lambda_1$ and $\lambda_2$ are both negative, then the singularity $(\T,\dot{\T})=(0,0)$ is a attractor node, represented in the Figure \ref{fig2}. Now, for $b^2<4a$ both eigenvalues $\lambda_1$ and $\lambda_2$ have the imaginary part different of zero and negative real part, then the singularity $(\T,\dot{\T})=(0,0)$ is a attractor focus, represented in the Figure \ref{fig3}. Both cases are topologically equivalent. Observe that without the damping effect, i.e. $b=0$, both eigenvalues $\lambda_1$ and $\lambda_2$ would be purely imaginary, then the singularity $(\T,\dot{\T})=(0,0)$ would be a center, represented in the Figure \ref{fig4}. This last case is not topologically equivalent to the two cases above.

\begin{figure}[h]
\psfrag{X}{$\T$}
\psfrag{Y}{$\dot{\T}$}
\includegraphics[width=5cm]{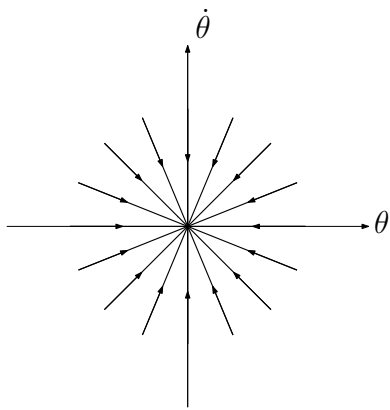}
\vskip 0cm \centerline{} \caption{\small \label{fig2} Attractor Node.}
\end{figure}

\begin{figure}[h]
\psfrag{X}{$\T$}
\psfrag{Y}{$\dot{\T}$}
\includegraphics[width=5cm]{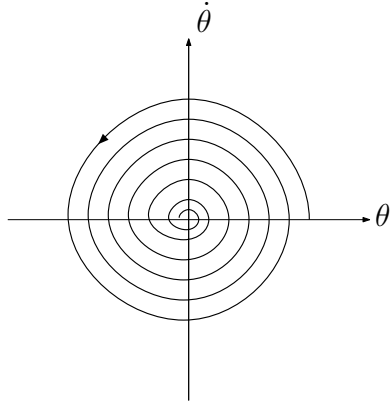}
\vskip 0cm \centerline{} \caption{\small \label{fig3} Attractor Focus.}
\end{figure}

\begin{figure}[h]
\psfrag{X}{$\T$}
\psfrag{Y}{$\dot{\T}$}
\includegraphics[width=5cm]{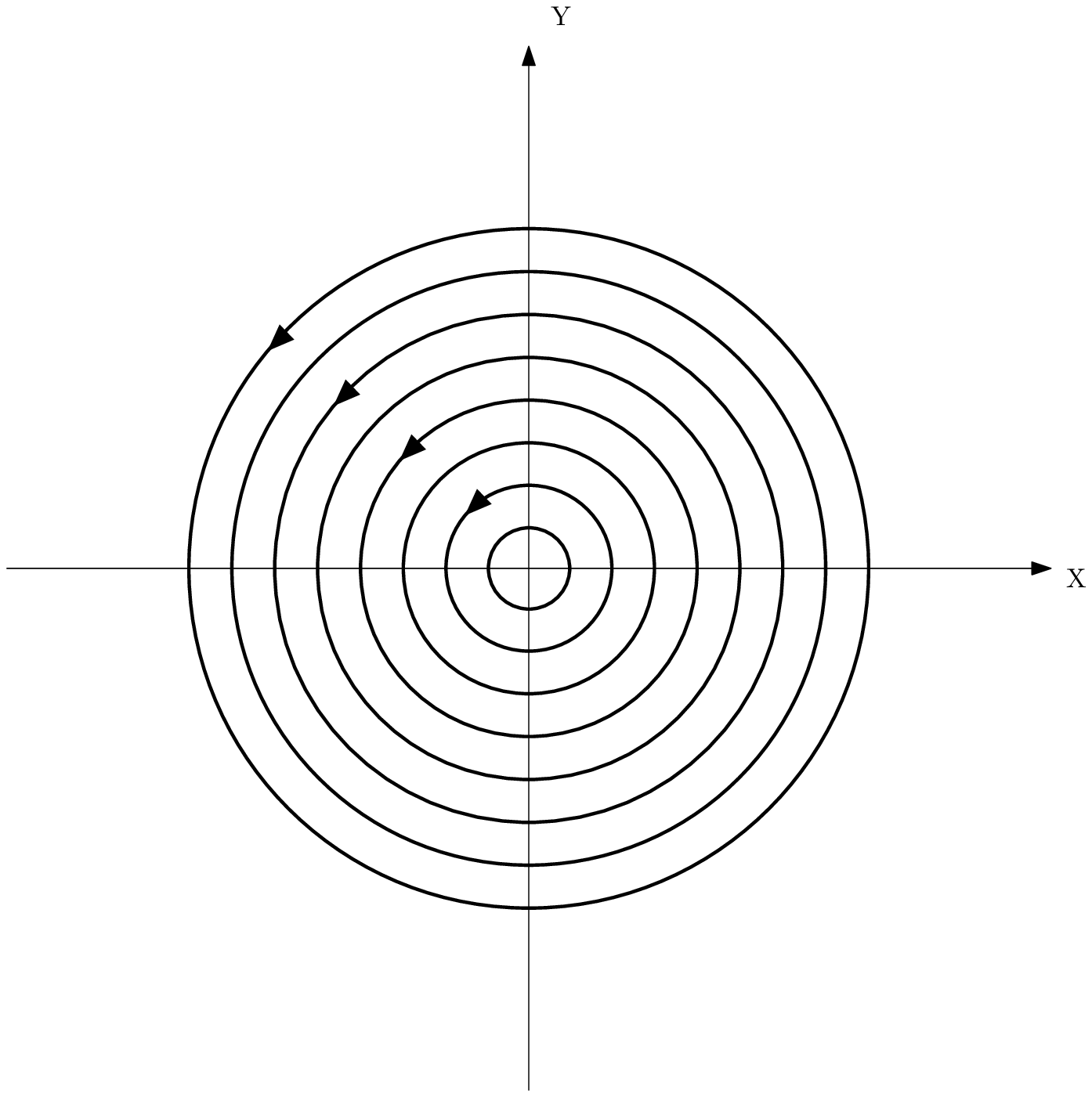}
\vskip 0cm \centerline{} \caption{\small \label{fig4} Center.}
\end{figure}

\smallskip

We emphasize here that, for $b\neq 0$, the linearized equation \eqref{sl1} can only be used to study the local behavior, at $(\T,\dot{\T})=(0,0)$, of the original (non--linear) system \eqref{s1}. However, a periodic solution of a differential system is a global element of the phase portrait. Thus to study these elements we have to consider the non--linear equation, which has the global information of the system.

\smallskip

As we have seen, when $b\neq 0$, the origin is an attractor singularity, thus the orbits of the system starting sufficiently closer to the origin tends to it. In other words, the damped pendulum always stops. The following studies provide conditions to drop this regime obtaining thus a periodic solution of the system which never reaches the origin.

\section{Perturbed System}\label{secPS}

Now, consider a perturbation of the system \eqref{s1} given by
\begin{equation}\label{s1a}
\ddot{\T}=-a\sin(\T)-b\dot{\T}+\e f(t,\T,\dot{\T})+\e^2g(t,\T,\dot{\T}),
\end{equation}
where the smooth functions $f(t,\T,\dot{\T})$ and $g(t,\T,\dot{\T})$ are respectively $T_f$ and $T_g$ periodic in the variable $t$ with
\[
\begin{array}{CCC}
T_f=\dfrac{p_f}{q_f}\dfrac{2\pi}{\sqrt{a}} & \textrm{and} & T_g=\dfrac{p_g}{q_g}\dfrac{2\pi}{\sqrt{a}},
\end{array}
\]
being $p$ and $q$ relatively prime positive integers for $p=p_f,\,p_g$ and $q=q_f,\,q_g$.

\smallskip

Our main goal is to find sufficient conditions for the functions $f$ and $g$ to assure the existence of periodic solutions of the system \eqref{s1a}.

\begin{remark}\label{r1}
For simplicity, we can assume, without loss of generality, that the functions $f(t,\T,\dot{\T})$ and $g(t,\T,\dot{\T})$ are $T$--periodic with $T=2p\pi/\sqrt{a}$ for some integer $p$. Indeed, if we take $p$ the least common multiple between $p_f$ and $p_g$, then there exists integers $n_f$ and $n_g$ such that $p=n_f\,p_f=n_g\,p_g$. Hence
\[
p\dfrac{2\pi}{\sqrt{a}}=n_f\,q_f\dfrac{p_f}{q_f}\dfrac{2\pi}{\sqrt{a}}=n_g\,q_g\dfrac{p_g}{q_g}\dfrac{2\pi}{\sqrt{a}}.
\]
\end{remark}

As we shall see in the Section \ref{sec3} the main result from the {\it Averaging Theory}, used in this present work, assumes that the perturbed system is given in an {\it Standard Form} \eqref{av1}. To get it, we have to introduce some changes of coordinates. The first change is done in the following.

\smallskip

If we take $\T=\e\phi$ with $|\e|>0$, then the system \eqref{s1a} becomes
\begin{equation}\label{s1b}
\ddot{\phi}=-a\dfrac{\sin(\e\phi)}{\e}-b\dot{\phi}+f(t,\e\phi,\e\dot{\phi})+\e g(t,\e\phi,\e\dot{\phi}).
\end{equation}

\smallskip

Since $b$ is a small parameter, we can assume that $b=\e\bar{b}$ with $\bar{b}>0$ if we consider perturbation for $\e>0$; and with $\bar{b}<0$ if we consider perturbation for $\e<0$. Henceforward we assume that $\e>0$.

\smallskip

Proceeding with Taylor series expansion near $\e=0$ in \eqref{s1b} we have that
\begin{equation}\label{s1c}
\ddot{\phi}=-a\phi+\e\left(g_0(t)+f_1(t)\phi+(f_2(t)-\bar{b})\dot{\phi}\right)+\e^2 r(t,\phi,\dot{\phi},\e),
\end{equation}
where
\[
\begin{array}{cccc}
g_0(t)=g(t,0,0), & f_1(t)=\dfrac{\p f}{\p\phi}(t,0,0), & \textrm{and} &f_2(t)=\dfrac{\p f}{\p\dot{\phi}}(t,0,0).
\end{array}
\]

\smallskip

Now, denoting $(x,y)=(\phi,\dot{\phi})$, the second order differential equation \eqref{s1c} can be written as the first order differential system
\begin{equation}\label{s2}
\begin{aligned}
\dot{x}&=y,\\
\dot{y}&=-ax+\e\left(g_0(t)+f_1(t)x+(f_2(t)-\bar{b})y\right)+\e^2 r(t,x,y,\e).
\end{aligned}
\end{equation}

\smallskip

Observe that the unperturbed system, i.e., $\e=0$, has the following eigenvalues
\[
\begin{array}{ccc}
\lambda_1=i\sqrt{a}& \textrm{and} & \lambda_2=-i\sqrt{a},
\end{array}
\]
thus it is a center at the origin, see Figure \ref{fig4}.

\smallskip

There are many works which deal with perturbation of centers, even in higher dimensions. For instance we can see the paper of Llibre and Teixeira \cite{LT2}.

\smallskip

The second change of coordinates \eqref{cg} will be done In the proof of Theorem \ref{t1}, which has its statement done in the next section.

\section{Statements of the main results}\label{sec2}

The objective of this paper is to provide an algebraic non--homogeneous linear system such that its solution provides a periodic solution of the
perturbed damped pendulum \eqref{s1a} for $\e>0$ sufficiently small.

\smallskip

Given $\z=(x_0,y_0)\in\R^2$, we define the linear system
\begin{equation}\label{S2}
M\z=v,
\end{equation}
where $M=(M_{ij})_{2\times 2}$ is a $2\times 2$ matrix, and $v$ is a vector such that
\[
\begin{array}{L}
M_{11}=-\dfrac{\bar b\pi}{\sqrt{a}}+\int_0^{\frac{2p\pi}{\sqrt{a}}}\sin(\sqrt{a}t)\left(-\dfrac{\cos(\sqrt{a}t)}{\sqrt{a}}f_1(t)+\sin(\sqrt{a}t)f_2(t)\right)dt,\\ \\

M_{12}=\int_0^{\frac{2p\pi}{\sqrt{a}}}\dfrac{\sin(\sqrt{a}t)}{a}\left(-\sin(\sqrt{a}t)f_1(t)-\sqrt{a}\cos(\sqrt{a}t)f_2(t)\right)dt,\\ \\

M_{21}=\int_0^{\frac{2p\pi}{\sqrt{a}}}\cos(\sqrt{a}t)\left(\cos(\sqrt{a}t)f_1(t)-\sqrt{a}\sin(\sqrt{a}t)f_2(t)\right)dt,\\ \\

M_{22}=-\dfrac{\bar b\pi}{\sqrt{a}}+\int_0^{\frac{2p\pi}{\sqrt{a}}}\cos(\sqrt{a}t)\left(\dfrac{\sin(\sqrt{a}t)}{\sqrt{a}}f_1(t)+\cos(\sqrt{a}t)f_2(t)\right)dt,
\end{array}
\]
and
\[
v=\left(
\begin{array}{C}
\int_0^{\frac{2p\pi}{\sqrt{a}}}\dfrac{\sin(\sqrt{a}t)}{\sqrt{a}}g_0(t)dt\\
-\int_0^{\frac{2p\pi}{\sqrt{a}}}\cos(\sqrt{a}t)g_0(t)dt
\end{array}
\right).
\]

\smallskip

Our main result on the periodic solutions of the damped pendulum with small perturbation \eqref{s1a} is the following.

\begin{theorem}\label{t1}
Assume that the functions $f$ and $g$ are smooth, $T$--periodic in the variable $t$ with $T=2p\pi/\sqrt{a}$ and $f(t,0,0)=0$. If $\det(M)\neq0$ and $v\neq(0,0)$, then for $\e> 0$ sufficiently small the perturbed damped pendulum \eqref{s1a} has a $T$--periodic solution $\T(t,\e)$, such that $(\T(0,\e),\dot{\T}(0,\e))\to (0,0)$ when $\e\to 0$.
\end{theorem}

Theorem \ref{t1} is proved in section \ref{sec4}. Its proof is based
in the averaging theory for computing periodic solutions, which will be introduce in section \ref{sec3}.

\smallskip

We provide an application of Theorem \ref{t1} in the following
corollary.

\begin{corollary}\label{c1}
Suppose that 
\[
\begin{array}{CCC}
\dfrac{\p f}{\p\phi}(t,0,0)=C_1, & \dfrac{\p f}{\p\dot{\phi}}(t,0,0)=C_2,
\end{array}
\]
\[
\begin{array}{CCC}
\int_0^{\frac{2\pi}{\sqrt{a}}}\sin(\sqrt{a}t)g(t,0,0)dt=V_1 & \textrm{and} & \int_0^{\frac{2\pi}{\sqrt{a}}}\cos(\sqrt{a}t)g(t,0,0)dt=V_2,
\end{array}
\]
with $(C_1,C_2)\neq(0,\bar b)$ and $(V_1,V_2)\neq (0,0)$. Then there exists a $2\pi/\sqrt{a}$--periodic solution $\T(t,\e)$ of the perturbed damped pendulum \eqref{s1a}, such that $(\T(0,\e),\dot{\T}(0,\e))\to (0,0)$ when $\e\to 0$.

\end{corollary}

The Corollary \ref{c1} will be proved in section \ref{sec4}.

\section{Averaging Theory}\label{sec3}

We present in this section a basic result known as \textit{First Order Averaging Theorem}. For a general introduction to averaging theory see for instance the book of Sanders and Verhulst \cite{SV} and the book of Verhulst \cite{V}.

\smallskip

We consider the differential equation
\begin{equation}\label{av1}
\dot{X}=\e\,F_1(s,X)+\e^2\,R(s,X,\e),
\end{equation}
where the smooth functions $F_1:\R\times U\rightarrow\R^n$ and $R:\R\times U\times(-\e_f,\e_f)\rightarrow\R^n$ are $T$-periodic in the first variable and $U$ is an open subset of $\R^n$.

\smallskip

We define the averaged system associated to system \eqref{av1} as
\begin{equation}\label{av3}
\dot{Z}(t)=\e f_1(Z(s)),
\end{equation}
where $f_1:U\rightarrow\R^n$ is given by
\[
f_1(Z)=\int_0^T F_1(s,Z)ds.
\]

\smallskip

The next theorem associates the singularities of the system \eqref{av3} with the periodic solutions of the differential system \eqref{av1}.

\begin{theorem}\label{t2}
Assume that for $a\in U$ with $f_1(a)=0$, there exist a neighborhood  $V$ of $a$ such that $f_1(Z)\neq0$ for all $z\in\bar{V}-\{a\}$ and $\det(df_1(a))\neq0$. Then, for $|\e|>0$ sufficiently small, there exist a $T$-periodic solution $X(t,\e)$ of the system \eqref{av1} such that $X(0,\e)\to a$ as $\e\to0$.
\end{theorem}

Using Brower degree theory, the hypotheses of Theorem \ref{t2} becomes weaker. For a proof of Theorem \ref{t2} see Buica and Llibre \cite{BL}.

\section{Proofs of Theorem \ref{t1} and Corollary \ref{c1}}\label{sec4}

In order to use the Theorem \ref{t2} in the proof of Theorem \ref{t1}, we have to modify the equation \eqref{s2}. If we denote
\[
\begin{array}{cccc}
\x=\begin{pmatrix}x\\y\end{pmatrix},& A=\begin{pmatrix}0&1\\-a&0\end{pmatrix}& F(t,\x)=\begin{pmatrix}0\\ g_0(t)+f_1(t)x+(f_2(t)-\bar{b})y\end{pmatrix},
\end{array}
\]
and
\[
R_0(t,\x)=\begin{pmatrix}0\\ r(t,\x,\e)\end{pmatrix},
\]
then the equation \eqref{s2} can be written in the matrix form
\begin{equation}\label{s4}
\dot{\x}=A\x+\e F(t,\x) +\e^2R_0(t,\x).
\end{equation}

\smallskip

Now, we proceed with the coordinates change taking $\y(t)\in\R^2$ as
\begin{equation}\label{cg}
\y(t)=e^{-At}\x(t),
\end{equation}
where $e^{At}$ is the matrix of the fundamental solution of the unperturbed differential system \eqref{s4}, i.e., $\e=0$.  This change of coordinates is done in the book of Sanders and Verhulst \cite{SVb}.

\smallskip

Observe that the function $t\mapsto e^{-At}$ is periodic with period equal $2\pi/\sqrt{a}$, since the eigenvalues of $A$ are purely imaginary. Moreover $y(0)=x(0)$.

\smallskip

The system \eqref{s4} is written in the new variable $\y$ as
\begin{equation}\label{s5}
\dot{\y}=\e \tilde{F}(t,\y)+\e^2\tilde{R}(t,\y,\e),
\end{equation}
with $\tilde{F}(t,\y)=e^{-At}F(t,e^{At}\y)$ and $\tilde{R}(t,\y,\e)=e^{-At}R_0(t,e^{At}\y,\e)$. Observe that the system \eqref{s5} is written in the standard form \eqref{av1}.

\smallskip

Now, we are ready to prove the Theorem \ref{t1}.

\begin{proof}[Proof of Theorem \ref{t1}]
The smooth functions $f(t,\x)$ and $g(t,\x)$ are $T$-periodic in the variable $t$, with $T=2p\pi/\sqrt{a}$, which implies that the smooth functions $\tilde{F}(t,\y)$ and $\tilde{R}(t,\y,\e)$ are also $T$-periodic in $t$.

\smallskip

We shall apply the Theorem \ref{t2} to the differential equation \eqref{s5}. Note that the equation \eqref{s5} can be written as the equation \eqref{av1} taking 
\[
\begin{array}{cccc}
X=\y,& F_1(t,X)=\tilde{F}(t,\y) &\textrm{and}& R(t,X,\e)=\tilde{R}(t,\y,\e).
\end{array}
\]

\smallskip

Given $\z\in\R^2$, we define the averaged function $f_1:\R^2\rightarrow \R^2$
\[
f_1(\z)=\int_0^{\frac{2p\pi}{\sqrt{a}}} e^{-At}F(t,e^{At}\z)dt.
\]

\smallskip

Proceeding with calculations we conclude that 
\[
f_1(\z)=M\z-v
\]
where $M$ and $v$ are defined in \eqref{S2}.

\smallskip

Assuming that $\det(M)\neq 0$ and $v\neq (0,0)$, we conclude that there exists a solution $\z_0=(x_0,y_0)\neq(0,0)$ of the linear system $f_1(\z_0)=0$ given by $\z_0=M^{-1} v$. 

\smallskip

Hence, applying Theorem \ref{t2}, follows that there exists a $T$--periodic solution $\y(t,\e)$ of the system \eqref{s5} such that
\[
\y(0,\e)\to z_0\quad \textrm{when } \e\to 0,
\]
which implies the existence of a periodic solution $(x(t,\e),y(t,\e))$ of the system \eqref{s2} such that
\[
(x(t,\e),y(t,\e))=\x(t,\e)=e^{At}\y(t,\e).
\]

\smallskip

Since $\x(0)=\y(0)$ it follows that
\[
\begin{pmatrix}
\displaystyle x(0,\e)\\
\displaystyle y(0,\e)
\end{pmatrix} \to
\begin{pmatrix}
\displaystyle x_0\\
\displaystyle y_0
\end{pmatrix},
\]
when $\e\to 0$. Hence 
\[
\left(\T(t,\e),\dot{\T}(t,\e)\right)=\e (x(t,\e),y(t,\e))
\]
is a $T$--periodic solution of the system \eqref{s1a} such that
\[
\left(\T(t,\e),\dot{\T}(t,\e)\right)\to (0,0),
\]
when $\e\to 0$.
\end{proof}

\begin{proof}[Proof of Corollary \ref{c1}]
The hypotheses of Corollary \ref{c1} implies that
\[
M=
\left(
\begin{array}{CC}
\dfrac{(C_2-\bar b)\pi}{\sqrt{a}}&-\dfrac{C_1\pi}{\sqrt{a^3}}\\
\dfrac{C_1\pi}{\sqrt{a}}&\dfrac{(C_2-\bar b)\pi}{\sqrt{a}}
\end{array}
\right).
\]
Thus
\[
\det(M)=\dfrac{(C_1^2+a(\bar b-C_2)^2)\pi^2}{a^2}.
\]
Hence $\det(M)=0$ if and only if $(C_1,C_2)=(0,\bar b)$. Since $(C_1,C_2)\neq(0,\bar b)$ and $v\neq(0,0)$, applying Theorem \ref{t1} the result follows.
\end{proof}

\section{Simulation}
Consider the following parameters in Corollary \ref{c1}: $a=b=1$, $C_1=C_2=0$; and the function $g(t,0,0)=\sin(t)$. Thus we have that $\det(M)=\pi^2$ and $V=(-1,0)$. Then, by Corollary \ref{c1}, there exists a periodic solution of the system \eqref{s1a}. 

\smallskip

Indeed, proceeding with the numerical simulation we find this periodic solution, see Figure \ref{SIMULA}.

\begin{figure}[h]
\psfrag{X}{t}
\psfrag{Y}{$\dot{\T}(t)$}
\psfrag{Z}{$\T(t)$}
\includegraphics[width=13cm]{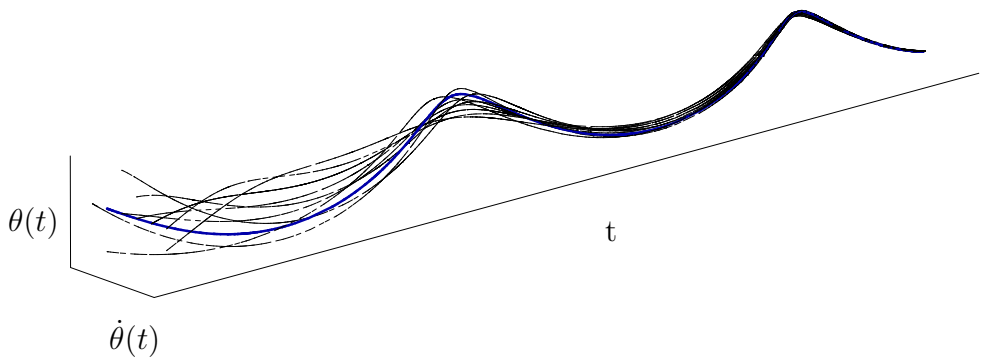}
\vskip 0cm \centerline{} \caption{\small \label{SIMULA} Solutions converging to the periodic solution.}
\end{figure}

\smallskip

This simulation has been done using the Wolfram Mathematica$^\circledR$ software.

\section{Conclusions and Future Directions}

The averaging theory is a collection of techniques to study, via approximations, the behavior of the solutions of a dynamical system under small perturbations. As we have seen, it can also be used to find periodic solutions.

\smallskip

In this paper, we have presented one of these techniques and used it to find conditions that assure the existence of a periodic solution of the perturbed damped pendulum system. We have got an algebraic non--homogeneous linear system such that its solution is associated with a periodic solution of the perturbed damped pendulum system.

\smallskip

Recently, Llibre, Novaes and Teixeira \cite{LNT2} have extended the averaging method for studying the periodic solutions of a class of differential equations with discontinuous second member. Therefore we are able to consider for instance equations of kind
\[
\ddot{\T}=-a\sin(\T)-b\,\sgn(\dot{\T})+\e f(t,\T,\dot{\T})+\e^2g(t,\T,\dot{\T}).
\]
Here, the term $b\,\sgn(\dot{\T})$ represents the Coulomb Friction and the function ${\rm sgn}(z)$ denotes the sign function, i.e.
\[
{\rm sgn}(z)=\left\{
\begin{array}{cl}
-1 & \mbox{if $z<0$,}\\
0 & \mbox{if $z=0$,}\\
1 & \mbox{if $z>0$.}
\end{array}
\right.
\]

\smallskip

For instance, in \cite{LNT3}, Llibre, Novaes and Teixeira have used the averaging theory to provide sufficient conditions for the existence of periodic
solutions with small amplitude of the non--linear planar double pendulum perturbed by smooth or non--smooth functions.

\section{Acknowledgements}
The  author is supported by a FAPESP--BRAZIL grant 2012/10231--7.


\begin{thebibliography}{99}

\bibitem{D} {\sc M. Denny},
{\it The pendulum clock: a venerable dynamical system}, {European Journal of Physics}, {\bf 23} (2002), 449--458.

\bibitem{LT} {\sc J. Llibre and M. A. Teixeira},
{\it On the stable limit cycle of a weight--driven pendulum clock}, {European Journal of Physics}, {\bf 31} (2010), 1249--1254.

\bibitem{MCPAP} {\sc G.A. Monerat, E.V. Corr\^{e}a Silva, G. Oliveira-Neto, A.R.P. de Assump\c{c}\~{a}o and A.R.R. Papa},
{\it Exploring hamiltonian systems: analytical study}, {Revista Brasileira de Ensino de F\'{i}sica}, {\bf 28}, n. 2, (2006), 177--189.

\bibitem{LNT1} {\sc J. Llibre, D.D. Novaes and M.A. Teixeira},
{\it On the periodic solutions of a perturbed double pendulum}, S\~{a}o Paulo J. Math Sciences IME--USP {\bf 5} (2011), 317--330.

\bibitem{AVK} \textsc{A. A. Andronov, A.A. Vitt and S. E. Khaikin},
\textit{Theory of oscillators}, International Series of Monographs In Physics \textbf{4}, Pergamon Press, 1966.

\bibitem{AP} \textsc{D. K. Arrowsmith and C. M. Place},
\textit{An introduction to dynamical system}, Cambridge University Press, 1990.

\bibitem{LT2} {\sc J. Llibre and M. A. Teixeira},
{\it Limit cycles for a mechanical system coming from the perturbation of a four-–dimensional linear center}, {Journal of Dynamical and Differential Equations}, {\bf 18}, n. 4, (2006), 931--941.

\bibitem{SV} {\sc J. A. Sanders and F. Verhulst},
{\it Averaging Methods in Nonlinear Dynamical Systems}, {Applied Mathematical Sciences vol 59}, {Berlin: Springer}, (1985).

\bibitem{V} {\sc F. Verhulst}, {\it Nonlinear dierential equations and dynamical systems}, Universitext, Springer, 1991.

\bibitem{BL} {\sc A. Buic\u{a} and J. Llibre},
{\it Averaging methods for finding periodic orbits via
Brouwer degree}, {Bull. Sci. Math.}, {\bf 128} (2004), 7--22.

\bibitem{SVb} {\sc J. A. Sanders and F. Verhulst},
{\it Averaging Methods in Nonlinear Dynamical Systems}, {Applied Mathematical Sciences vol 59}, {Berlin: Springer}, (1985).

\bibitem{LNT2} {\sc J. Llibre, D.D. Novaes and M.A. Teixeira},
{\it Averaging methods for studying the periodic orbits of discontinuous differential systems}, arXiv:1205.4211 [math.DS], http://arxiv.org/abs/1205.4211.

\bibitem{LNT3} {\sc J. Llibre, D.D. Novaes and M.A. Teixeira},
{\it On the periodic solutions of a generalized smooth or non-smooth perturbed planar double pendulum}, arXiv:1203.0498 [math.DS], http://arxiv.org/abs/1203.0498.

\end{thebibliography}
\end{document}